\documentclass[12pt,reqno]{amsart}
\usepackage{amsmath}
\usepackage{amssymb}
\usepackage{amstext}
\usepackage{a4wide}
\usepackage{graphicx}
\allowdisplaybreaks \numberwithin{equation}{section}
\usepackage{color}
\usepackage{cases}

\numberwithin{equation}{section}

\newtheorem{theorem}{Theorem}[section]

\newtheorem{lemma}[theorem]{Lemma}

\theoremstyle{definition}

\newtheorem{definition}[theorem]{Definition}

\theoremstyle{remark}
\newtheorem{remark}[theorem]{Remark}

\begin{document}

\title
[Location of concentrated vortices]{Location of concentrated vortices in planar steady Euler flows}

 \author{Guodong Wang,  Bijun Zuo}
\address{Institute for Advanced Study in Mathematics, Harbin Institute of Technology, Harbin 150001, P.R. China}
\email{wangguodong@hit.edu.cn}
\address{College of Mathematical Sciences, Harbin Engineering University, Harbin {\rm150001}, PR China}
\email{bjzuo@amss.ac.cn}


\begin{abstract}
In this paper, we study two-dimensional steady incompressible Euler flows in which the vorticity is sharply concentrated in a finite number of regions of small diameter in a bounded domain. Mathematical analysis of such flows is an interesting and physically important research topic in fluid mechanics.
The main purpose of this paper is to prove that in such flows the locations of these concentrated blobs of vorticity must be in the vicinity of some critical point of the Kirchhoff-Routh function, which is determined by the geometry of the domain. The vorticity is assumed to be only in $L^{4/3},$ which is the optimal regularity for weak solutions to make sense. As a by-product, we prove a nonexistence result for concentrated multiple vortex flows in convex domains.
\end{abstract}

\maketitle
\section{Introduction}
Let $D\subset\mathbb R^2$ be a simply-connected bounded domain with smooth boundary $\partial D$. Consider  in $D$ an ideal fluid in steady state, the motion of which is described by the famous Euler equations
\begin{equation}\label{euler}
\begin{cases}
(\mathbf v\cdot\nabla)\mathbf v=-\nabla P&\mathbf x=(x_1,x_2)\in D,\\
\nabla\cdot\mathbf v=0&\mathbf x\in D,\\
\mathbf v\cdot\mathbf n =g&\mathbf x\in\partial D,
\end{cases}
\end{equation}
where $\mathbf v=(v_1,v_2)$ is the velocity field, $P$ is a scalar function that represents the pressure, $\mathbf n$ is the unit outward normal on $\partial D,$ and $g$ is a given function satisfying the following compatibility condition
\begin{equation}\label{g}
\int_{\partial D}gdS=0.
\end{equation}
Here we assume that the fluid is of unit density. The first two equations in \eqref{euler} are the momentum conservation and mass conservation respectively, and the boundary condition in \eqref{euler} means that the rate of mass flow across the boundary per unit area is $g$. In particular, if $g\equiv0,$ then there is no matter flow through the boundary.

The scalar vorticity $\omega$, defined as the signed magnitude of curl$\mathbf v,$ that is,
\[\omega=\partial_{x_1} v_2-\partial_{x_2}v_1,\]
is one of the fundamental physical quantities and plays an important role in the study of two-dimensional flows. 

Below we reformulate the Euler equations \eqref{euler} as a single equation of $\omega$, which is much easier to handle mathematically. 
First we show that $\mathbf v$ can be recovered from $\omega$. In fact, since $\mathbf v$ is divergence-free and $D$ is simply-connected, we can apply the Green's theorem to show that there is a scalar function $\psi,$ called the \emph{stream function}, such that
\begin{equation}\label{psi}
\mathbf v=(\partial_{x_2}\psi,-\partial_{x_1}\psi).
\end{equation}
For convenience, throughout this paper we will use the symbol
$\mathbf b^\perp$ to denote the clockwise rotation through $\pi/2$ of any planar vector $\mathbf b=(b_1,b_2)$, that is, $\mathbf b^\perp=(b_2,-b_1)$, and $\nabla^\perp f$ to denote $(\nabla f)^\perp$ for any scalar function $f$, that is, $\nabla^\perp f=(\partial_{x_2}f,-\partial_{x_1}f)$. Thus \eqref{psi} can also be written as
\begin{equation}\label{psi2}
\mathbf v=\nabla^\perp \psi.
\end{equation}
It is easy to check that $\psi$ and $\omega$ satisfy
\begin{equation}\label{poisson}
\begin{cases}
-\Delta\psi=\omega&\text{in }D,\\
\nabla^\perp\psi\cdot\mathbf n=g&\text{on }\partial D.
\end{cases}
\end{equation}
To deal with the boundary condition in \eqref{poisson}, we consider the following elliptic problem
\begin{equation}\label{q}
\begin{cases}
-\Delta \psi_0=0&\text{in }D,\\
\nabla^\perp \psi_0\cdot\mathbf n=g&\text{on }\partial D.
\end{cases}
\end{equation}
To solve \eqref{q}, we  first solve the following Laplace equation with standard Neumann boundary condition
\begin{equation*}
\begin{cases}
-\Delta \psi_1=0&\text{in }D,\\
\frac{\partial \psi_1}{\partial\mathbf n}=g&\text{on }\partial D,
\end{cases}
\end{equation*}
then the harmonic conjugate of $\psi_1$ solves \eqref{q}. Note that by the maximum principle the solution to \eqref{q} is unique up to a constant.
Now it is easy to see that $\psi-\psi_0$ satisfies
\begin{equation}\label{gw}
\begin{cases}
-\Delta(\psi-\psi_0)=\omega&\text{in }D,\\
\nabla^\perp (\psi-\psi_0)\cdot\mathbf n=0&\text{on }\partial D.
\end{cases}
\end{equation}
The boundary condition in \eqref{gw} implies that $\psi-\psi_0$ is a constant on $\partial D$ (recall that $D$ is simply-connected). Without loss of generality by adding a suitable constant we assume that $\psi-\psi_0=0$ on $\partial D,$ thus $\psi-\psi_0$ can be expressed in terms of the Green's operator as follows
\begin{equation}\label{exp}
\psi-\psi_0=\mathcal G\omega:=\int_DG(\cdot,\mathbf y)\omega(\mathbf y)d\mathbf y,
\end{equation}
where $G(\cdot,\cdot)$ is the Green's function for $-\Delta$ in $D$ with zero boundary condition. Combining \eqref{psi2} and \eqref{exp}, we have recovered $\psi$ from $\omega$ in the following
\begin{equation}\label{bs}
\mathbf v=\nabla^\perp(\mathcal G\omega+\psi_0),
\end{equation}
which is usually called the Biot-Savart law in fluid mechanics.  
On the other hand, taking the curl on both sides of the momentum equation in \eqref{euler} we get
\begin{equation}\label{ve1}
\mathbf v\cdot\nabla \omega=0.
\end{equation}
From \eqref{bs} and \eqref{ve1}, the Euler equations \eqref{euler} are reduced to a single equation of $\omega$ 
\begin{equation}\label{ve}
\nabla^\perp(\mathcal G\omega+\psi_0)\cdot\nabla \omega=0\quad\text{ in }D,
\end{equation}
which is usually called the \emph{vorticity equation}.

\begin{remark}
When $D$ is multiply-connected, the above discussion is still valid. The only difference is that one needs to replace the usual Green's function $G$ by the hydrodynamic Green's function (see \cite{Flu}, Definition 15.1), which does not cause any essential difficulty for the problem discussed in this paper.
\end{remark}

In the rest of this paper, we will be focused on the study of \eqref{ve}. Note that once we have obtained a solution $\omega$ to \eqref{ve}, we immediately get a solution to \eqref{euler} with
\[\mathbf v=\nabla^\perp(\mathcal G\omega+\psi_0),\quad P(\mathbf x)=\int_{L_{\mathbf x_0,\mathbf x}}\omega(\mathbf y)\mathbf v^\perp(\mathbf y)\cdot d\mathbf y-\frac{1}{2}|\mathbf v(x)|^2,\]
where $\mathbf x_0$ is a fixed point in $D$ and $L_{\mathbf x_0,\mathbf x}$ is any $C^1$ curve joining $\mathbf x_0$ and $\mathbf x$ (one can easily check that the above line integral is well defined by using Green's theorem and the fact that $\omega$ is a solution).

Since in many physical problems the vorticity is of low regularity, not even continuous,  it is necessary to define the notion of weak solutions to \eqref{ve}. In the rest of this paper, we regard $\psi_0$ as a given function.
\begin{definition}\label{wsve}
Let $\omega\in L^{4/3}(D)$. If for any $\phi\in C_c^\infty(D)$ it holds that
\begin{equation}\label{int}
    \int_D\omega\nabla^\perp(\mathcal G\omega+\psi_0)\cdot\nabla\phi d\mathbf x=0,\end{equation}
then $\omega$ is called a weak solution to the vorticity equation \eqref{ve}.
\end{definition}
The above definition is reasonable from the fact that one can multiply any test function 
$\phi$ on both sides of \eqref{ve} and integrate by parts formally to get \eqref{int}.

\begin{remark}
Since $\psi_0$ is harmonic (thus smooth) and $\phi$ has compact support in $D$, we see that the integral $\int_D\omega \nabla^\perp \psi_0\cdot\nabla\phi d\mathbf x$ in \eqref{int} makes sense. Note that throughout this paper we do not impose any condition on the boundary value of $\psi_0$. 
\end{remark}

\begin{remark}
By the Calderon-Zygmund inequality and Sobolev inequality, $\omega\in L^{4/3}(D)$ is  the optimal regularity for the integral $\int_D\omega\nabla^\perp\mathcal G\omega\cdot\nabla\phi d\mathbf x$ in \eqref{int} to be well-defined.
\end{remark}

In the literature, there has been extensive study on the existence of weak solutions to  \eqref{ve}. See \cite{B1,B2,CLW,CPY1,CPY2,CW1,CWZ2,CWZu,EM,LP,LYY,SV,T,W,WZ} for example. The solutions obtained in these papers have one common feature, that is, the vorticity is the function of the stream function ``locally". In this regard, Cao and Wang \cite{CW2} proved a general criterion for an $L^{4/3}$ function to be a weak solution.
\begin{theorem}[Cao--Wang, \cite{CW2}]\label{cwthm}
Let $k$ be a positive integer and $\psi_0\in C^1(\bar D)$.
Suppose that $\omega\in L^{4/3}(D)$ satisfies
\begin{equation}\label{fo}
\omega=\sum_{i=1}^k\omega_i, \,\,\min_{1\leq i< j\leq k}\{\text{dist}(\text{supp}\omega_i,\text{supp}\omega_j)\}>0,\,\,\omega_i=f^i(\mathcal G\omega+\psi_0), \text{a.e. in  (supp}\omega_i)_\delta,
\end{equation}
where $\delta$ is a positive number, 
\[\text{(supp}\omega_i)_\delta:=\{\mathbf x\in D\mid \text{dist}(\mathbf x,\text{supp}\omega_i)<\delta\},\]
and each $f^i$ is either monotone from $\mathbb R$ to $\mathbb R\cup\{\pm\infty\}$ or Lipschitz from $\mathbb R$ to $\mathbb R$.
Then $\omega$ is a weak solution to the vorticity equation \eqref{ve}.
\end{theorem}

Some examples of such flows are as follows. When $f_i$ in Theorem \ref{cwthm} is a Heaviside function, the solutions are called vortex patches, and related existence results can be found in \cite{CPY1,CW1,CWZu,T,WZ}. When $f_i$ is a power function, related papers are \cite{CLW,CPY2,LP,LYY,SV}. In \cite{B1,B2,EM}, the authors obtained some steady vortex flows by maximizing or minimizing the kinetic energy of the fluid on the rearrangement class of some given function. The solutions obtained in \cite{B1,B2,EM} still have the form \eqref{fo}, where each $f_i$ is a monotone function, but the precise expression of $f_i$ is unknown. Recently Cao, Wang and Zhan \cite{CWZ2,W}  modified Turkington's method \cite{T} and  proved the existence of a large class of solutions of the form \eqref{fo}, where each $f_i$ is a given function with few restrictions.

Among the flows mentioned above, some are of particular interest and attract more attention, that is, flows in which the vorticity is sharply concentrated in a finite number of small regions and vanishes elsewhere, just like a finite sum of Dirac measures. Mathematically, the vorticity in such flows has the form
\begin{equation}\label{cv}
\omega_\varepsilon=\sum_{i=1}^k\omega_{\varepsilon,i},\quad 
{\rm supp}(\omega_{\varepsilon,i})\subset B_{o(1)}(\bar x_i),\quad\int_D\omega_{\varepsilon,i} d\mathbf x=\kappa_i+o(1),\quad i=1,\cdot\cdot\cdot,k,
\end{equation}
where $\varepsilon$ is a small positive parameter, $k$ is a positive integer, $\bar x_i\in D$, $\kappa_i$ is a fixed non-zero real number, $i=1,\cdot\cdot\cdot,k$, and $o(1)\to0$ as $\varepsilon \to0^+$. Papers concerning the existence of such solutions include \cite{CLW,CPY1,CPY2,CW1,CWZ2,CWZu,SV,T,W}. Note that all the flows constructed in these papers have bounded vorticity.

Euler Flows with vorticity of the form \eqref{cv} is closely related to a very famous Hamiltonian system in $\mathbb R^2$, the point vortex model (see \cite{L}), which describes the evolution of a finite number of point vortices with their locations being the canonical variables. The point vortex model is only an approximate model, and its precise connection with the 2D Euler equations with concentrated vorticity in the evolutionary case is a tough and unsolved problem. For a detailed discussion, we refer the interested readers to  \cite{MP1,MP2,MP3,MPa,T2}. According to the point vortex model, the locations of concentrated blobs of vorticity in steady Euler flows are not arbitrary, but \emph{should} be near a critical point of the following Kirchhoff-Routh function
\begin{equation}\label{krf}
W(\mathbf x_1,\cdot\cdot\cdot,\mathbf x_k)=-\sum_{1\leq i<j\leq k}\kappa_i\kappa_jG(\mathbf x_i,\mathbf x_j)+\frac{1}{2}\sum_{i=1}^k\kappa_i^2H(\mathbf x_i)+\sum_{i=1}^k\kappa_i\psi_0(\mathbf x_i),
\end{equation}
where 
\[(\mathbf x_1,\cdot\cdot\cdot,\mathbf x_k)\in \underbrace{D\times\cdot\cdot\cdot\times D}_{k\text { times}}\setminus \{(\mathbf x_1,\cdot\cdot\cdot,\mathbf x_k)\mid \mathbf x_i\in D, \mathbf x_i=\mathbf x_j \text{ for some }i\neq j\}\]
and $H(\mathbf x)=h(\mathbf x,\mathbf x)$ with $h$ being the regular part of Green's function, that is,
\[h(\mathbf x,\mathbf y):=-\frac{1}{2\pi}\ln|\mathbf x-\mathbf y|-G(\mathbf x,\mathbf y),\quad \mathbf x,\mathbf y\in D.\]
However, to our knowledge there is no complete and rigorous proof on this issue in the literature, although the solutions of the form \eqref{cv} constructed in \cite{CLW,CPY1,CPY2,CW1,CWZ2,EM,SV,T} are all based on the hypothesis that $(\bar {\mathbf x}_1,\cdot\cdot\cdot,\bar {\mathbf x}_k)$ is a critical point of $ W$.
The aim of paper is prove that such a hypothesis is necessary.

This paper is organized as follows. In Section 2, we state our main results (Theorems \ref{mthm} and \ref{none}) and give some comments. In Sections 3 and 4 we provide the proofs of them.

\section{Main results}
In this section, we present our two main results. 
The first result is about the necessary condition about the locations of concentrated vortices.
 
 \begin{theorem}\label{mthm}
 Let  $k$ be a positive integer, $\bar {\mathbf x}_1,\cdot\cdot\cdot,\bar {\mathbf x}_k\in D$ be $k$ different points and $\kappa_1,\cdot\cdot\cdot,\kappa_k$ be $k$ non-zero real numbers.
Assume that there exists a sequence of weak solutions $\{\omega_n\}_{n=1}^{+\infty}$ to the vorticity equation \eqref{ve}, satisfying 
$\omega_n=\sum_{i=1}^k\omega_{n,i}$ with  $\omega_{n,i}\in L^{4/3}(D)$  and
\[{\rm supp}(\omega_{n,i})\subset B_{o(1)}(\bar {\mathbf x}_i),\quad
\int_D\omega_{n,i} d\mathbf x=\kappa_i+o(1),\quad i=1,\cdot\cdot\cdot,k,\]
where $o(1)\to0$ as $n\to+\infty$. 
Then $(\bar {\mathbf x}_1,\cdot\cdot\cdot,\bar {\mathbf x}_k)$ must be a critical point of $W$ defined by \eqref{krf}.
 \end{theorem}

Here we  compare Theorem \ref{mthm} with two related results in \cite{CGPY} and \cite{CM}.
 In \cite{CGPY}, Cao, Guo, Peng and Yan studied planar Euler flows with vorticity of the following patch form 
 \begin{equation}\label{ceq}
 \omega^\lambda=\sum_{i=1}^k\omega^\lambda_i,\quad
 \omega^\lambda_i=\lambda\chi_{\{\mathbf x\in D\mid \mathcal G\omega^\lambda(\mathbf x)>\mu_i^\lambda\}\cap B_{\delta}(\bar {\mathbf x}_i)},\quad \int_D\omega_i^\lambda d\mathbf x=\kappa_i, \quad i=1,\cdot\cdot\cdot,k,
 \end{equation}
 where $\lambda$ is a large positive parameter, $\chi$ denotes the characteristic function, each $\mu^\lambda_i$ is a real number depending on $\lambda$ and each $\kappa_i$ is a given non-zero number.
They proved that if supp$\omega^\lambda_i$ ``shrinks" to $\bar {\mathbf x}_i$ as $\lambda\to+\infty$, then $\bar {\mathbf x}_1\cdot\cdot\cdot,\bar {\mathbf x}_k$ must necessarily constitute a critical point of $W$ (see Theorem 1.1 in \cite{CGPY} for the precise statement) . Compared with their result, we consider more general flows and only impose very weak regularity on the vorticity in Theorem \ref{mthm}. Moreover, as we will see in the next section, the proof we provide is shorter and more elementary.  The other relevant work is \cite{CM}. In \cite{CM}, Caprini and Marchioro studied the evolution of a finite number of blobs of vorticity in $\mathbb R^2$ and  proved the finite-time localization property (see Theorem 1.2 in \cite{CM} for the precise statement). In their result, each $\omega_{n,i}$ is required additionally to have a definite sign and satisfy the growth condition
 \begin{equation}\label{gwth}
 \|\omega_{n,i}\|_{L^\infty}\leq M(\text{diam(supp}\omega_{n,i}))^{-\delta},
 \end{equation}
where $M$ and $\delta$ are both fixed positive numbers. As a consequence of their result, Theorem \ref{mthm} holds true if the additional growth condition \eqref{gwth} is satisfied (although they only considered the whole plane case, similar result for a bounded domain can also be proved without any difficulty). In this sense, our result can be regarded as a strengthened version of Caprini and Marchioro's result in the steady case.

\begin{remark}
In Theorem 1.1 in \cite{CGPY}, for vorticity of the form \eqref{ceq}, $\bar{\mathbf x}_i\in D$ and $\bar {\mathbf x}_i\neq \bar {\mathbf x}_j$ for $ i\neq j$ are not assumptions but can be proved as conclusions.  However, in the very general setting of this paper, these two conclusions may be false. For example, we can regard
 a single blob of vorticity as two artificially, thus they may concentrate on the same point. Also,  Cao, Wang and Zuo \cite{CWZu} constructed a pair steady vortex patches with opposite rotation directions in the unit disk (Theorem 5.1, \cite{CWZu}), and it can be checked that as the ratio of the circulations of the two patches goes to infinity,  the patch with smaller circulation will approach the boundary of the disk.
 \end{remark}

Our second result is about the nonexistence of concentrated multiple vortex flows in convex domains, which can be seen as a by-product of Theorem \ref{mthm}.

\begin{theorem}\label{none}
Let  $\delta_0>0$ be fixed, $D$ be a smooth convex domain, $k\geq 2$ be a positive integer, $ \kappa_1,\cdot\cdot\cdot,\kappa_k$ be $k$ positive numbers and $f_1,\cdot\cdot\cdot,f_k$ be $k$ real functions satisfying
\[\lim_{t\to0^+}f_i(t)=0,\quad i=1\cdot\cdot\cdot,k.\]
If $\psi_0\equiv0,$ then there exists $\varepsilon_0>0$, such that for any 
$\varepsilon\in(0,\varepsilon_0),$ there is no weak solution $\omega_\varepsilon$ to the vorticity equation \eqref{ve} satisfying
\begin{itemize}
\item[(1)] $\omega_\varepsilon=\sum_{i=1}^k\omega_{\varepsilon,i},$ $\omega_{\varepsilon,i}\in L^{4/3}(D), i=1\cdot\cdot\cdot,k;$
\item[(2)] $\text{dist(supp}\omega_{\varepsilon,i},\text{supp}\omega_{\varepsilon,j}) >\delta_0 \,\, \forall\,1\leq i<j\leq k$ and
$\text{dist(supp}\omega_{\varepsilon,i},\partial D)>\delta_0\,\, \forall\,1\leq i\leq k;$
\item[(3)] diam(supp$\omega_{\varepsilon,i}$)<$\varepsilon,\,\,i=1,\cdot\cdot\cdot,k.$
\item[(4)] $\int_D\omega_{\varepsilon,i}d\mathbf x=\kappa_i+f_i(\varepsilon),\,\,i=1,\cdot\cdot\cdot,k.$
\end{itemize}

\end{theorem}

\section{Proof of Theorem \ref{mthm}}

First we need the following lemma.

\begin{lemma}\label{lem}
Let $\omega\in L^{4/3}(\mathbb R^2)$ with compact support. Define
\begin{equation*}
f(\mathbf x)=\int_{\mathbb R^2}\ln|\mathbf x-\mathbf y|\omega(\mathbf y)d\mathbf y.
\end{equation*}
Then $f\in W^{2,4/3}_{\rm loc}(\mathbb R^2)$ and the distributional partial derivatives of $f$ can be expressed as
\begin{equation}\label{deri}
\partial_{x_i} f(\mathbf x)=\int_{\mathbb R^2}\frac{x_i-y_i}{|\mathbf x-\mathbf y|^2}\omega(\mathbf y)d\mathbf y\quad \text{a.e. }\,\mathbf x\in\mathbb R^2, \,\,i=1,2.
\end{equation}
\end{lemma}

\begin{proof}
By the Calderon-Zygmund estimate we have $f\in W^{2,4/3}_{\rm loc}(\mathbb R^2)$. The expression \eqref{deri} follows from Theorem 6.21 on page 157, \cite{LL}.

\end{proof}

Now we are ready to prove Theorem \ref{mthm}. The key point of the proof is to use the anti-symmetry of the singular part of the Biot-Savart kernel.

\begin{proof}[Proof of Theorem \ref{mthm}]
Fix $l\in\{1,\cdot\cdot\cdot,k\}$. It is sufficient to show that
\[\nabla_{{\mathbf x}_l}W(\bar {\mathbf x}_1,\cdot\cdot\cdot,\bar {\mathbf x}_k)=\mathbf 0.\]

Let $r_0$ be a small positive number such that
\[r_0<\text{dist}(\bar{\mathbf x}_i,\partial D)\quad \forall\,1\leq i\leq k,\quad r_0<\frac{1}{2}\text{dist}(\bar{\mathbf x}_i,\bar{\mathbf x}_j)\quad \forall\,1\leq i<j\leq k.\]
Choose $\phi(\mathbf x)=\rho(\mathbf x)\mathbf b\cdot \mathbf x$ in Definition \ref{wsve}, where $\mathbf b$ is a constant planar vector and $\rho$ satisfies
\[\rho\in C_c^\infty( D),\,\,\rho\equiv 1\text{ in } B_{r_0}(\bar {\mathbf  x}_l),\,\,\rho\equiv 0\text{ in } B_{r_0}(\bar {\mathbf  x}_i)\,\,\forall\,i\neq l.\]
Existence of such $\rho$ can be easily obtained by mollifying a suitable patch function.
Then we have
\[\int_D\omega_n\nabla^\perp\left(\mathcal G\omega_n+\psi_0\right)\cdot\nabla\phi d\mathbf x
=0.\]
Denote
\[A_n=\int_D\omega_n\nabla^\perp\mathcal G\omega_n\cdot\nabla\phi d\mathbf x,\quad B_n=\int_D\omega_n\nabla^\perp\psi_0\cdot\nabla\phi d\mathbf x.\]
Then \begin{equation}\label{ab}
A_n+B_n=0,\quad n=1,2,\cdot\cdot\cdot.
\end{equation}
Below we analyze $A_n$ and $B_n$ separately.

For $A_n$, we have
\begin{align*}
 A_n=&\int_D\omega_n(\mathbf x)\nabla^\perp_{\mathbf x}\int_D\left(-\frac{1}{2\pi}\ln|\mathbf x-\mathbf y|-h(\mathbf x,\mathbf y)\right)\omega_n(\mathbf y)d\mathbf y\cdot\nabla\phi d\mathbf x\\ 
 =&-\frac{1}{2\pi}\int_D\omega_n(\mathbf x)\int_D\frac{(\mathbf x-\mathbf y)^\perp}{|\mathbf x-\mathbf y|^2}\omega_n(\mathbf y)d\mathbf y\cdot\nabla\phi d\mathbf x-\int_D\omega_n\int_D\nabla^\perp_{\mathbf x}h(\mathbf x,\mathbf y)\omega_n(\mathbf x)(\mathbf y)d\mathbf y\cdot\nabla\phi d\mathbf x.
 \end{align*}
 Here we used Lemma \ref{lem} and the facts that $h\in C^\infty(D\times D)$ and $\omega_n$ has compact support in $D$.
Since $\omega_n\in L^{4/3}(D)$, by the Hardy-Littlewood-Sobolev inequality (see Theorem 0.3.2 in \cite{SO}) we have
 \[\int_D\frac{|\omega_n(\mathbf y)|}{|\mathbf x-\mathbf y|}d\mathbf y\in L^4(D).\]
Now we can apply Fubini's theorem (see \cite{ru}, page 164) to obtain
\[\frac{(\mathbf x-\mathbf y)^\perp}{|\mathbf x-\mathbf y|^2}\cdot\nabla\phi\omega_n(\mathbf x)\omega_n(\mathbf y)\in L^1(D\times D)\]
and
\[\int_D\omega_n(\mathbf x)\int_D\frac{(\mathbf x-\mathbf y)^\perp}{|\mathbf x-\mathbf y|^2}\omega_n(\mathbf y)d\mathbf y\cdot\nabla\phi d\mathbf x=\int_D\int_D\frac{(\mathbf x-\mathbf y)^\perp\cdot\nabla\phi}{|\mathbf x-\mathbf y|^2}\omega_n(\mathbf x)\omega_n(\mathbf y) d\mathbf xd\mathbf y.\]
Thus we have obtained
 \begin{align*}
 A_n=-\frac{1}{2\pi}\int_D\int_D\frac{(\mathbf x-\mathbf y)^\perp\cdot\nabla\phi}{|\mathbf x-\mathbf y|^2}\omega_n(\mathbf x)\omega_n(\mathbf y) d\mathbf xd\mathbf y-\int_D\int_D\nabla^\perp_{\mathbf x}h(\mathbf x,\mathbf y)\cdot\nabla\phi\omega_n(\mathbf x) \omega_n(\mathbf y)d\mathbf xd\mathbf y.
 \end{align*}
 Substituting $\phi(\mathbf x)=\rho(\mathbf x)\mathbf b\cdot\mathbf x$ in $A_n$, for sufficiently large $n$ we have
  \begin{align*}
 A_n&=-\frac{1}{2\pi}\int_D\int_{D}\frac{(\mathbf x-\mathbf y)^\perp\cdot\mathbf b}{|\mathbf x-\mathbf y|^2}\omega_{n,l}(\mathbf x)\omega_n(\mathbf y) d\mathbf xd\mathbf y-\int_D\int_D\nabla^\perp_{\mathbf x}h(\mathbf x,\mathbf y)\cdot\mathbf b\omega_{n,l}(\mathbf x) \omega_n(\mathbf y)d\mathbf xd\mathbf y\\
  =&-\frac{1}{2\pi}\sum_{j=1}^k\int_D\int_D\frac{(\mathbf x-\mathbf y)^\perp \cdot\mathbf b}{|\mathbf x-\mathbf y|^2}\omega_{n,l}(\mathbf x)\omega_{n,j}(\mathbf y)d\mathbf xd\mathbf y
 -\sum_{j=1}^k\int_D\int_D\nabla_{\mathbf x}^\perp h(\mathbf x,\mathbf y)\cdot\mathbf b\omega_{n,l}(\mathbf x)\omega_{n,j}(y)d\mathbf xd\mathbf y\\
 =&-\frac{1}{2\pi}\int_D\int_D\frac{(\mathbf x-\mathbf y)^\perp \cdot\mathbf b}{|\mathbf x-\mathbf y|^2}\omega_{n,l}(\mathbf x)\omega_{n,l}(\mathbf y)d\mathbf xd\mathbf y
 -\frac{1}{2\pi}\sum_{j=1,j\neq l}^k\int_D\int_D\frac{(\mathbf x-\mathbf y)^\perp \cdot\mathbf b}{|\mathbf x-\mathbf y|^2}\omega_{n,l}(\mathbf x)\omega_{n,j}(\mathbf y)d\mathbf xd\mathbf y\\
 &-\sum_{j=1}^k\int_D\int_D\nabla_{\mathbf x}^\perp h(\mathbf x,\mathbf y)\cdot\mathbf b\omega_{n,l}(\mathbf x)\omega_{n,j}(\mathbf y)d\mathbf xd\mathbf y\\
 =&-\frac{1}{2\pi}\int_D\int_D\frac{(\mathbf x-\mathbf y)^\perp \cdot\mathbf b}{|\mathbf x-\mathbf y|^2}\omega_{n,l}(\mathbf x)\omega_{n,l}(\mathbf y)d\mathbf xd\mathbf y
 +\sum_{j=1,j\neq l}^k\int_D\int_D\nabla_{\mathbf x}^\perp G(\mathbf x,\mathbf y)\cdot\mathbf b\omega_{n,l}(\mathbf x)\omega_{n,j}(\mathbf y)d\mathbf xd\mathbf y\\
 &-\int_D\int_D\nabla_{\mathbf x}^\perp h(\mathbf x,\mathbf y)\cdot\mathbf b\omega_{n,l}(\mathbf x)\omega_{n,l}(\mathbf y)d\mathbf xd\mathbf y\\
 :=&C_n+D_n,
\end{align*}
where
\[C_n=-\frac{1}{2\pi}\int_D\int_D\frac{(\mathbf x-\mathbf y)^\perp \cdot\mathbf b}{|\mathbf x-\mathbf y|^2}\omega_{n,l}(\mathbf x)\omega_{n,l}(\mathbf y)d\mathbf xd\mathbf y,\]
\[D_n=\left(\sum_{j=1,j\neq l}^k\int_D\int_D\nabla_{\mathbf x}^\perp G(\mathbf x,\mathbf y)\omega_{n,l}(\mathbf x)\omega_{n,j}(\mathbf y)d\mathbf xd\mathbf y
-\int_D\int_D\nabla_{\mathbf x}^\perp h(\mathbf x,\mathbf y)\omega_{n,l}(\mathbf \mathbf x)\omega_{n,l}(\mathbf y)d\mathbf xd\mathbf y\right)\cdot\mathbf b.\]
By the anti-symmetric property of the integrand in $C_n$, we see that 
\[C_n=0 \quad \text{for sufficiently large $n$}.\]
For $D_n,$ it is clear that
\[\lim_{n\to+\infty}D_n=\left(\sum_{j=1,j\neq l}^k \kappa_l\kappa_j\nabla_{\mathbf x}^\perp G(\bar {\mathbf x}_l,\bar {\mathbf x}_j)
-\kappa_l^2\nabla_{\mathbf x}^\perp h(\bar {\mathbf x}_l,\bar {\mathbf x}_l)\right)\cdot\mathbf b.\]
To conclude, we have obtained
\begin{equation}\label{a}
\lim_{n\to+\infty}A_n=\left(\sum_{j=1,j\neq l}^k \kappa_l\kappa_j\nabla_x^\perp G(\bar {\mathbf x}_l,\bar {\mathbf x}_j)
-\kappa_l^2\nabla_{\mathbf x}^\perp h(\bar {\mathbf x}_l,\bar {\mathbf x}_l)\right)\cdot\mathbf b.
\end{equation}
For $B_n,$ it is also clear that
\begin{equation}\label{b}
\lim_{n\to+\infty}B_n=\kappa_l\nabla^\perp\psi_0(\bar{\mathbf x}_l)\cdot\mathbf b.
\end{equation}

Combining \eqref{ab}, \eqref{a} and \eqref{b} we immediately get

\[\left(\sum_{j=1,j\neq l}^k \kappa_l\kappa_j\nabla_{\mathbf x}^\perp G(\bar {\mathbf x}_l,\bar {\mathbf x}_j)
-\kappa_l^2\nabla_{\mathbf x}^\perp h(\bar {\mathbf x}_l,\bar {\mathbf x}_l)+\kappa_l\nabla^\perp\psi_0(\bar{\mathbf x}_l)\right)\cdot \mathbf b= 0\]
for sufficiently large $n$.
Since $\mathbf b$ can be any constant vector, we deduce that
\[\sum_{j=1,j\neq l}^k \kappa_l\kappa_j\nabla_{\mathbf x}^\perp G(\bar {\mathbf x}_l,\bar {\mathbf x}_j)
-\kappa_l^2\nabla_{\mathbf x}^\perp h(\bar {\mathbf x}_l,\bar {\mathbf x}_l)+\kappa_l\nabla^\perp\psi_0(\bar{\mathbf x}_l)= \mathbf 0,\]
which is exactly
\[\nabla_{{\mathbf x}_l}W(\bar {\mathbf x}_1,\cdot\cdot\cdot,\bar {\mathbf x}_k)=\mathbf 0.\]

\end{proof}

\section{Proof of Theorem \ref{none}}

In this section we give the proof of Theorem \ref{none}. To begin with, we need an important property of the Kirchhoff-Routh function in a convex domain proved by Grossi and Takahashi. We only state the following simple version of their result which is enough for our use.

\begin{theorem}[Grossi--Takahashi, Theorem 3.2, \cite{GT}]\label{gtt}
Let $D$ be a smooth convex domain, $k\geq 2$ be a positive integer and $\kappa_1,\cdot\cdot\cdot,\kappa_k$ be $k$ positive numbers. If $\psi_0\equiv0,$ then the Kirchhoff-Routh function $W$ defined by \eqref{krf} has no critical point in
\[\underbrace{D\times\cdot\cdot\cdot\times D}_{k\text { times}}\setminus \{(\mathbf x_1,\cdot\cdot\cdot,\mathbf x_k)\mid \mathbf x_i\in D, \mathbf x_i=\mathbf x_j \text{ for some }i\neq j\}.\]
\end{theorem}

\begin{proof}[Proof of Theorem \ref{none}]
Suppose, by contradiction, that there exist a sequence of positive numbers $\{\varepsilon_{n}\}_{n=1}^{+\infty}$, $\varepsilon_n\to0^+$ as $n\to+\infty$, and a sequence of weak solutions $\{\omega_n\}_{n=1}^{+\infty}$ to the vorticity equation \eqref{ve} such that
\begin{itemize}
\item[(i)] $\omega_n=\sum_{i=1}^k\omega_{n,i},$ $\omega_{n,i}\in L^{4/3}(D), i=1\cdot\cdot\cdot,k;$
\item[(ii)] $\text{dist(supp}\omega_{n,i},\text{supp}\omega_{n,j}) >\delta_0 \,\,\,\forall\,1\leq i<j\leq k$ and
$\text{dist(supp}\omega_{n,i},\partial D)>\delta_0\,\,\, \forall\,1\leq i\leq k;$
\item[(iii)] diam(supp$\omega_{n,i}$)<$\varepsilon_n,\,\,i=1,\cdot\cdot\cdot,k.$
\item[(iv)] $\int_D\omega_{n,i}d\mathbf x=\kappa_i+f_i(\varepsilon_n),\,\,i=1,\cdot\cdot\cdot,k.$
\end{itemize}
Define
\[{\mathbf x}_{n,i}=\left(\int_{D}\omega_{n,i}d\mathbf x\right)^{-1}\int_{D}\mathbf x\omega_{n,i}d\mathbf x,\quad i=1\cdot\cdot\cdot,k.\]
By (ii) and (iii) we see that 
\[\text{dist}(\mathbf x_{n,i},\mathbf x_{n,j}) \geq\frac{\delta_0}{2} \,\,\, \forall\,1\leq i<j\leq k,\quad
\text{dist}(\mathbf x_{n,i},\partial D)\geq\frac{\delta_0}{2}\,\,\, \forall\,1\leq i\leq k\]
if $n$ is large enough.
Thus we can choose a subsequence $\{\mathbf x_{n_m,i}\}$ such that $\mathbf x_{n_m,i}\to\bar {\mathbf x}_i$ as $m\to+\infty, i=1,\cdot\cdot\cdot,k,$ where $\bar{\mathbf x}_1,\cdot\cdot\cdot,\bar{\mathbf x}_k$ satisfy
 \[\text{dist}(\bar{\mathbf x}_{i},\bar {\mathbf x}_{j}) \geq\frac{\delta_0}{2} \,\,\, \forall\,1\leq i<j\leq k,\quad
\text{dist}(\bar{\mathbf x}_{i},\partial D)\geq\frac{\delta_0}{2}\,\,\, \forall\,1\leq i\leq k.\]

Now we can see that the sequence of solutions$\{\omega_{n_m}\}$ satisfies the assumptions in Theorem \ref{mthm}, and therefore $(\bar{\mathbf x}_1,\cdot\cdot\cdot,\bar{\mathbf x}_k)$ must be a critical point of $W$ (with $\psi_0\equiv 0$). This is a contradiction to Theorem \ref{gtt}.

\end{proof}

\smallskip

{\bf Acknowledgements:}
{G. Wang was supported by National Natural Science Foundation of China (12001135, 12071098) and China Postdoctoral Science Foundation (2019M661261).}

\phantom{s}
 \thispagestyle{empty}

\end{document}